\documentclass[a4paper]{amsart}

\usepackage{enumerate}
 
\usepackage{mathtools}
\usepackage{amssymb}
\usepackage{cite}

\newcommand{\XN}{ \ensuremath{ X^{x,N} } }
\newcommand{\YN}{ \ensuremath{ Y^{x,N} } }
\newcommand{\ZN}{ \ensuremath{ Z^{x,N} } }
\newcommand{\YhatN}{ \ensuremath{ \hat{Y}^{x,N} } }
\newcommand{\eN}{ \ensuremath{ e^{x,N} } }

\newcommand{\dWN}{ \ensuremath{ \Delta W^N } }
\newcommand{\Wbar}{ \ensuremath{ \bar{W}^N } }

\newcommand{\oss}{ \ensuremath{ \Theta } }

\newcommand{\lerr}{ \ensuremath{ \Delta_\mathrm{loc} } }
\newcommand{\mil}{ \ensuremath{ \varphi_\mathrm{Mil} } }
\newcommand{\tmil}{ \ensuremath{ \oss_\mathrm{Mil} } }
\newcommand{\htil}{ \ensuremath{ \tilde{h} } }
\newcommand{\Ytmil}{ \ensuremath{ \bar{Y}^{x,N} } }
\newcommand{\exin}{ \ensuremath{\alpha_\mathrm{ex} } }

\newcommand{\N}{\mathbb{N}}

\newcommand{\R}{\mathbb{R}}

\newcommand{\eps}{\ensuremath{\varepsilon}}
\newcommand{\set}[1]{\ensuremath{ \left\{ #1 \right\} } }
\newcommand{\ind}[1]{\ensuremath{ 1_{ #1 } } }

\newcommand{\abs}[1]{\ensuremath{ \left\vert #1 \right\vert } }

\newcommand{\dd}{\mathrm{d}}

\newcommand{\E}[1]{\ensuremath{ \mathrm{E}\left( #1 \right) }}
\newcommand{\prob}{\ensuremath{ \mathrm{P} }}
\newcommand{\F}{\ensuremath{ \mathfrak{F} }}

\newcommand{\eqdist}{ \ensuremath{ \stackrel{\mathrm{d}}{=} } }

\theoremstyle{plain}
	\newtheorem{lem}{Lemma}
	\newtheorem{thm}{Theorem}
	\newtheorem{cor}{Corollary}
	\newtheorem{prop}{Proposition}
\theoremstyle{definition}

\theoremstyle{remark}
	\newtheorem{rem}{Remark}
	\newtheorem{ex}{Example}

\begin{document}

\title[Strong Convergence Rates for {CIR} Processes]
	{Strong Convergence Rates for
	{C}ox-{Ingersoll}-{R}oss Processes --
	Full Parameter Range}

\author[Hefter]{Mario Hefter}
\address{Fachbereich Mathematik\\
Technische Universit\"at Kaisers\-lautern\\
Postfach 3049\\
67653 Kaiserslautern\\
Germany}
\email{hefter@mathematik.uni-kl.de}

\author[Herzwurm]{Andr\'{e} Herzwurm}
\address{Fachbereich Mathematik\\
Technische Universit\"at Kaisers\-lautern\\
Postfach 3049\\
67653 Kaiserslautern\\
Germany}
\email{herzwurm@mathematik.uni-kl.de}

\begin{abstract}
	We study strong (pathwise) approximation of Cox-Ingersoll-Ross processes.
	We propose a Milstein-type scheme that is suitably truncated close to zero,
	where the diffusion coefficient fails to be locally Lipschitz continuous.
	For this scheme we prove polynomial convergence rates for the full parameter
	range including the accessible boundary regime. The error criterion is
	given by the maximal $L_p$-distance of the solution and its approximation
	on a compact interval. In the particular case of a squared
	Bessel process of dimension $\delta>0$ the polynomial convergence rate 
	is given by $\min(1,\delta)/(2p)$.
\end{abstract}



\maketitle

% ------------------------------------------------------------------------
% ------------------------------------------------------------------------
% Section 1
% ------------------------------------------------------------------------
% ------------------------------------------------------------------------
\section{Introduction}
In recent years, strong approximation has been intensively studied for
stochastic differential equations (SDEs) of the form
\begin{align}\label{eq:cir-intro}
	\dd X^x_t = (a-bX^x_t)\,\dd t + \sigma\sqrt{X^x_t}\,\dd W_t, \quad X^x_0=x, \quad t\geq0,
\end{align}
with a one-dimensional Brownian motion $W$, initial value $x\geq0$,
and parameters $a,\sigma>0$, $b\in\R$. It is well known that these SDEs
have a unique non-negative strong solution. Such SDEs often arise
in mathematical finance, e.g., as the volatility process in the
Heston model. Moreover, they were proposed as a model for
interest rates in the Cox-Ingersoll-Ross (CIR) model.
In the particular case of $b=0$ and $\sigma=2$ the solution of
SDE~\eqref{eq:cir-intro} is a squared Bessel process,
which plays in important role in the analysis of Brownian motion.

Strong approximation is of particular interest due to the multi-level
Monte Carlo technique, see \cite{giles1,giles2,heinrich1}.
In this context, a sufficiently high convergence rate with respect to
an $L_2$-norm is crucial. The multi-level Monte Carlo technique is
used for the approximation of the expected value of a functional applied
to the solution of an SDE. In mathematical finance, such a functional
often represents a discounted payoff of some derivative
and the price is then given by the corresponding expected value.

Strong convergence of numerical schemes for the SDE~\eqref{eq:cir-intro}
has been widely studied in the past twenty years. Various schemes
have been proposed and proven to be strongly convergent, see, e.g.,
\cite{alfonsi2005,deelstra1998,higham2005,gyoengy2011,milstein2015,hj2015}.
In recent years, the speed of convergence in terms of polynomial convergence
rates has been intensively studied, see
\cite{bessel1,berkaoui,dereich,alfonsi2013,neuenkirch-szpruch,hjn-cir,chassagneux,bossy}.
In all these articles the approximation error is measured with respect
to the $L_p$-norm either at a single time point,
at the discretization points, or on a compact interval.
Moreover, all these results impose conditions on $p$ (appearing in the $L_p$-norm)
and the quantity
\begin{align}\label{eq:delta}
	\delta = \frac{4a}{\sigma^2}  \in\left]0,\infty\right[,
\end{align}
which depends on the two parameters $a,\sigma>0$ appearing
in \eqref{eq:cir-intro}. None of these results
yield a polynomial convergence rate for $\delta<1$, cf.~Figure~\ref{fig:CIR-rates}.
By the Feller test, the solution remains strictly positive, i.e.,
$\prob(\forall t>0: X^x_t>0)=1$, if and only if $\delta\geq2$.
Roughly speaking, the smaller the value of $\delta$, the closer the
solution to the boundary point zero, where the diffusion coefficient
is not even locally Lipschitz continuous.

\begin{figure}[htbp]
	\centering
	\includegraphics[width=0.9\linewidth]{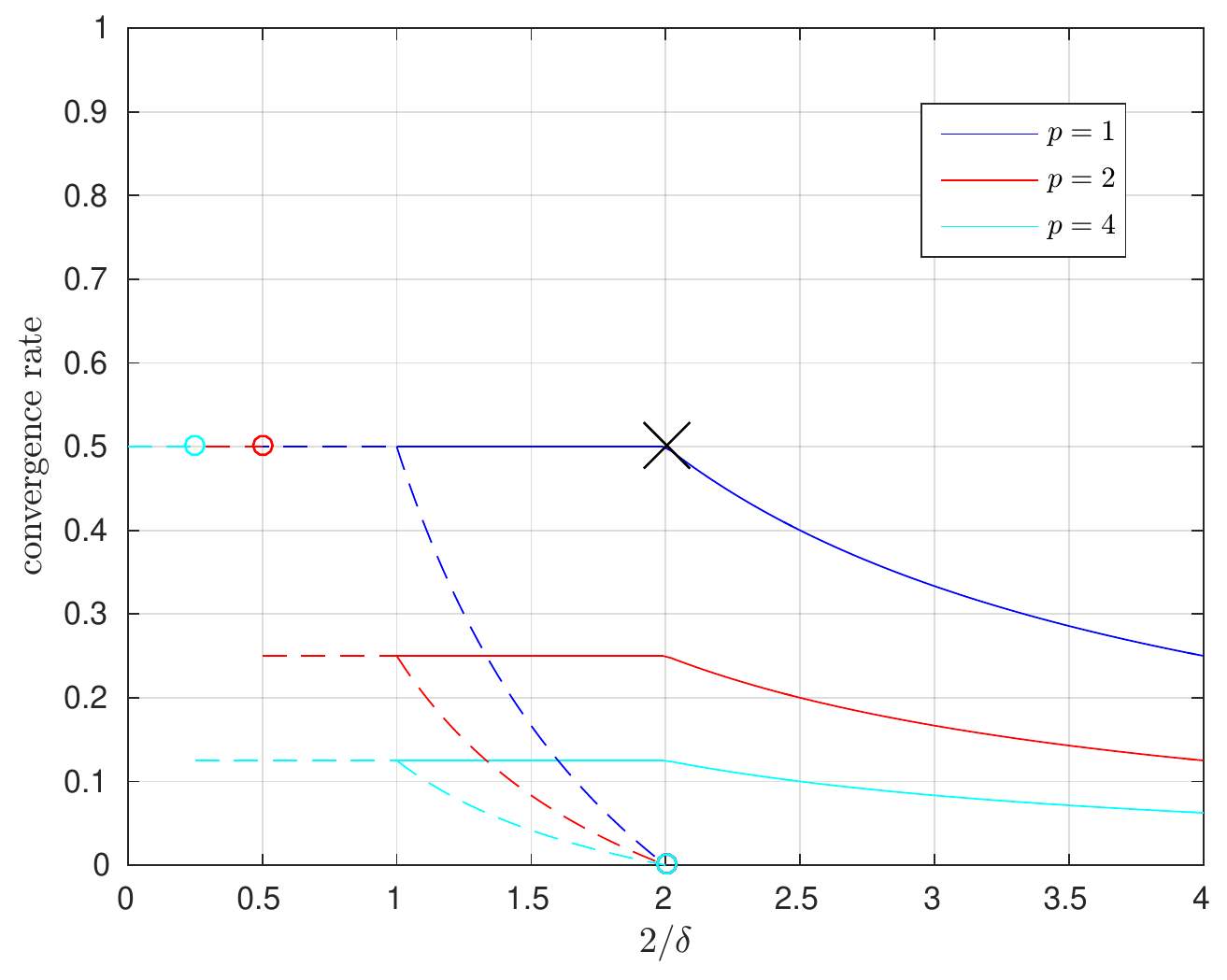}
	\caption{The colored dashed lines show the convergence rates of the known
		best upper bounds for the error criterion in \eqref{eq:intro} for
		different values of $p$, see \cite{dereich,hjn-cir}.
		The corresponding solid lines show the convergence rates obtained
		in this paper. The cross shows the convergence rate at $\delta=1$
		for all $p\in\left[1,\infty\right[$, see \cite{bessel1}.}
	\label{fig:CIR-rates}
\end{figure}

The main aim of this paper is to construct a numerical scheme
for the SDE~\eqref{eq:cir-intro} and to prove its convergence at a polynomial rate
for all parameters. Let $T>0$ and
define the approximation scheme $\Ytmil=(\Ytmil_t)_{0\leq t\leq T}$,
which uses $N\in\N$ increments of Brownian motion, by
\begin{align*}
	\Ytmil_0=x \qquad\text{and}\qquad
		\Ytmil_{(n+1)T/N}
		= \tmil\left( \Ytmil_{nT/N}, T/N, W_{(n+1)T/N}-W_{nT/N} \right),
\end{align*}
for $n=0,\dots,N-1$, where the one-step function
$\tmil\colon\R^+_0\times\left]0,T\right]\times\R\to\R^+_0$ is given by
\begin{align*}
	\tmil(x,t,w) = \left(
			\left( \max\left(
				\sqrt{{\sigma^2}/{4}\cdot t},\sqrt{\max(\sigma^2/4\cdot t,x)}+{\sigma}/{2}\cdot w
			\right) \right)^2
			+ \left( a-\sigma^2/4-b\cdot x \right)\cdot t
		\right)^+.
\end{align*}
This yields a discrete-time approximation of the SDE~\eqref{eq:cir-intro}
on $[0,T]$ based on a grid of mesh size $T/N$.
Furthermore, between two grid points we use constant interpolation to get
a continuous-time approximation, i.e.,
\begin{align*}
	\Ytmil_t = \Ytmil_{nT/N}, \qquad t\in{[nT/N,(n+1)T/N[},
\end{align*}
for $n=0,\dots,N-1$.
We refer to $\Ytmil$ as truncated Milstein scheme. Let us mention that this scheme
coincides with the classical Milstein scheme as long as it is ``away'' from zero.
Moreover, it is suitably truncated close to zero.

\begin{thm}[Main Result]\label{thm:intro}
	Let $\delta>0$ be according to \eqref{eq:delta}.
	For every $1\leq p<\infty$ and every $\eps>0$ there exists a constant $C>0$
	such that
	\begin{align}\label{eq:intro}
			\sup_{0\leq t\leq T}\left( \E{ \abs{X^x_t-\Ytmil_t}^p } \right)^{1/p}
				\leq C\cdot(1+x)\cdot \frac{ 1}{ N^{\min(1,\delta)/(2p)-\eps} }
	\end{align}
	for all $N\in\N$ and for all $x\geq0$.
\end{thm}

\begin{figure}[htbp]
	\centering
	\includegraphics[width=0.9\linewidth]{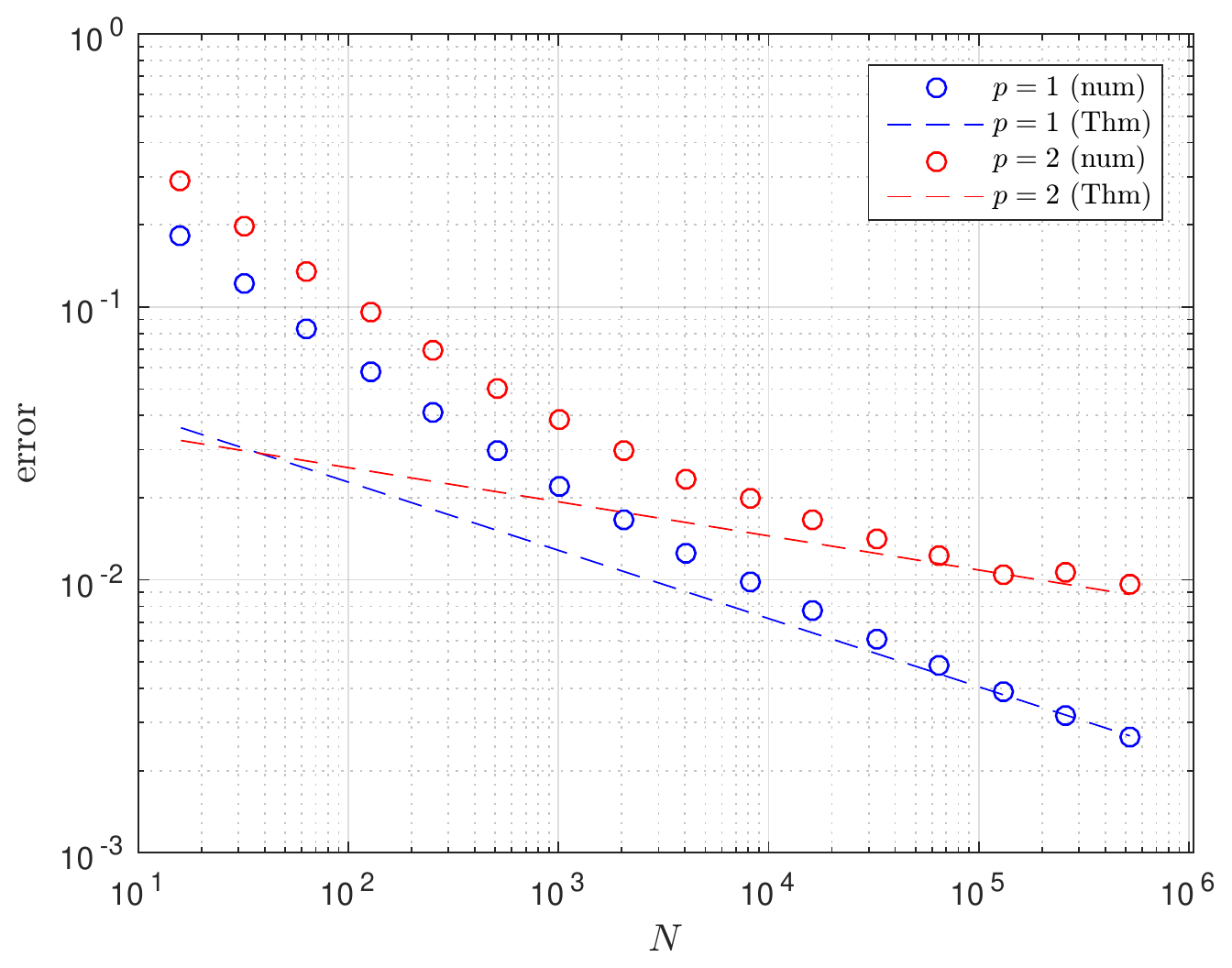}
	\caption{Numerical results for the truncated Milstein scheme
		with $\sigma=2$, $a=\delta=1/2$, $b=T=1$, and $x=1/20$ are shown for $p=1,2$.
		The error given by \eqref{eq:intro} is estimated based on $10^5$ replications
		of consecutive differences.
		The dashed lines show the corresponding convergence rates
		from Theorem~\ref{thm:intro}. For a numerical study of various Euler-
		and Milstein-type schemes (with $2/\delta<3$ and $p=1$) we refer to
		\cite[Fig.~3]{alfonsi2005}.
		}
	\label{fig:numerics}
\end{figure}

The results of Theorem~\ref{thm:intro} are illustrated in
Figure~\ref{fig:CIR-rates}. Observe that the convergence rate in
Theorem~\ref{thm:intro}, given by $\min(1,\delta)/(2p)$,
is monotonically decreasing in $1/\delta$ and $p$.
This is the same monotonicity as in previous results.
The numerical results with $\delta=1/2$ for the truncated Milstein scheme
presented in Figure~\ref{fig:numerics} indicate that the convergence rate
in Theorem~\ref{thm:intro} is sharp for $p=1$.
Furthermore, these results suggest that the $L_p$-norm affects the convergence rate,
as expected from Theorem~\ref{thm:intro}.
We think that the convergence rate in Theorem~\ref{thm:intro} is sharp for $p=1$
in the full parameter range but might be too pessimistic for $p>1$.
Note that a Milstein-type scheme, the drift-implicit Euler scheme, converges at rate
$1/2$ for all $p\geq1$ in the particular situation of $\delta=1$ and $b=0$, see \cite{bessel1}. 
Let us stress that a convergence rate of $1/2$ is optimal for various global error criteria
in case of SDEs satisfying standard Lipschitz assumptions, see, e.g., \cite{gronbach1,hofmann}.

Let us mention that Theorem~\ref{thm:intro} actually holds for a whole class of
approximation schemes. More precisely, a one-step approximation scheme satisfies the error bound
\eqref{eq:intro} if it is $L_1$-Lipschitz continuous, has a Milstein-type local
discretization error and is uniformly bounded, see
(A\ref{ass:lipschitz})-(A\ref{ass:bounded}) for the precise assumptions.
Our analysis is based on the $L_1$-norm since CIR processes are in general
not locally $L_p$-Lipschitz continuous in the initial value if $p>1$,
see Proposition~\ref{prop:initial-value}.
The results for $L_1$ are then extended to $p>1$ by using classical
interpolation techniques. Due to this interpolation the convergence
rate in Theorem~\ref{thm:intro} is divided by $p$.

This paper is organized as follows. In Section~\ref{sec:preliminaries} we
recall some basic properties of the solution of SDE~\eqref{eq:cir-intro}.
In Section~\ref{sec:strong} we provide a general framework for the analysis
of one-step approximation schemes and prove strong convergence rates under suitable
assumptions on such a one-step scheme. In Section~\ref{sec:scheme} we study the
truncated Milstein scheme and show that it satisfies the assumptions
of Section~\ref{sec:strong}.

% ------------------------------------------------------------------------
% ------------------------------------------------------------------------
% Section 2
% ------------------------------------------------------------------------
% ------------------------------------------------------------------------
\section{Notation}
We use $\N=\{1,2,3,\dots\}$, $\N_0=\N\cup\{0\}$, $\R^+=\{x\in\R: x>0\}$,
and $\R^+_0=\R^+\cup\{0\}$.
Moreover, $x^+=\max(0,x)$ denotes the positive part of $x\in\R$.
We use $\eqdist$ to denote equality in distribution.
We do not explicitly indicate if statements are only
valid almost surely, e.g., equality of two random variables.
For functions $f$ and $g$ taking non-negative values we write $f\preceq g$
if there exists a constant $C>0$ such that $f\leq C\cdot g$.
Moreover, we write $f\asymp g$ if $f\preceq g$ and $g\preceq f$.
From the context it will be clear which objects the constant $C$ is allowed
to depend on.

% ------------------------------------------------------------------------
% ------------------------------------------------------------------------
% Section 3
% ------------------------------------------------------------------------
% ------------------------------------------------------------------------
\section{Preliminaries}\label{sec:preliminaries}
In this section we set up the framework and provide some basic facts about the
SDE~\eqref{eq:cir-intro} that will be frequently used.
To simplify the presentation we only consider the SDE~\eqref{eq:cir-intro}
with $\sigma=2$. Moreover, we only study
approximation on the unit interval $[0,1]$ instead of $[0,T]$.
Both simplifications are justified by the following two remarks.

\begin{rem}[Reduction to $\sigma=2$]\label{rem:scaling-space}
	Consider the particular case of SDE~\eqref{eq:cir-intro} given by
	\begin{align*}
		\dd \hat{X}^{\hat{x}}_t = (\delta-b\hat{X}^{\hat{x}}_t)\,\dd t
			+ 2\,\sqrt{\hat{X}^{\hat{x}}_t}\,\dd W_t,
			\quad \hat{X}^{\hat{x}}_0=\hat{x}, \quad t\geq0,
	\end{align*}
	with $\hat x=x\cdot 4/\sigma^2$ and $\delta$ given by \eqref{eq:delta}.
	Then we have
	\begin{align*}
		X_t^x = \frac{\sigma^2}{4}\cdot \hat X^{\hat x}_t
	\end{align*}
	for all $t\geq 0$.
\end{rem}

\begin{rem}[Reduction to $T=1$]\label{rem:scaling-time}
	Consider the instance of SDE~\eqref{eq:cir-intro} given by
	\begin{align*}
		\dd \tilde{X}^{x}_t = (\tilde{a}-\tilde{b} \tilde{X}^{x}_t)\,\dd t
			+ \tilde{\sigma}\,\sqrt{\tilde{X}^{x}_t}\,\dd \tilde{W}_t,
			\quad \tilde{X}^{x}_0={x}, \quad t\geq0,
	\end{align*}
	with $\tilde a=T\cdot a$, $\tilde b=T\cdot b$, $\tilde\sigma=\sqrt{T}\cdot\sigma$,
	and Brownian motion $\tilde W=(\tilde W_t)_{t\geq 0}$ given by
	\begin{align*}
		\tilde W_t=1/\sqrt{T}\cdot W_{t\cdot T}.
	\end{align*}
	Then we have
	\begin{align*}
		X^x_{t} = \tilde X^x_{t/T}
	\end{align*}
	for all $t\geq 0$.
\end{rem}

Throughout the rest of the paper $(\Omega,\mathfrak A,\prob)$ denotes the
underlying probability space that is assumed to be complete and $\F=(\F_t)_{t\geq 0}$
denotes a filtration that satisfies the usual conditions, i.e., $\F_0$ contains
all $\prob$-null-sets and $\F$ is right-continuous.
Moreover, $W=(W_t)_{t\geq 0}$ denotes a scalar Brownian motion with respect to $\F$.
Finally, the process $X^x=(X^x_t)_{t\geq 0}$ with initial value $x\geq 0$ is given by the SDE
\begin{align}\label{eq:cir}
	\dd X^x_t = (\delta-bX^x_t)\,\dd t + 2\sqrt{X^x_t}\,\dd W_t, \quad X^x_0=x, \quad t\geq0,
\end{align}
with fixed parameters $\delta>0$ and $b\in\R$.

\begin{rem}
Let us mention that
\begin{align}\label{eq:mean-sol}
	\E{X^x_t} = x\cdot e^{-bt} + \delta\cdot
		\begin{cases}
			(1-e^{-bt})/b, &\text{if }\ b\neq0, \\
			t, &\text{if }\ b=0,
		\end{cases}
\end{align}
for $t\geq0$, see, e.g., \cite{cir}.
\end{rem}

\begin{rem}[Marginal distribution of CIR processes]\label{rem:marginal}
	Let $Z$ be $\chi^2$-distributed with $\delta$ degrees of freedom,
	i.e., $Z$ admits a Lebesgue density that is proportional to
	\begin{align*}
		z\mapsto z^{\delta/2-1}\cdot \exp(-z/2)\cdot \ind{\R^+}(z).
	\end{align*}
	If $b=0$, we have
	\begin{align*}
		X^0_t \eqdist t\cdot X^0_1 \eqdist t\cdot Z
	\end{align*}
	for $t\geq0$, according to
	\cite[Prop.~XI.1.6]{revuz-yor} and \cite[Cor.~XI.1.4]{revuz-yor}.
	In the general case $b\in\R$ we obtain
	\begin{align}\label{eq:marginal}
		X^0_t \eqdist \psi(t)\cdot Z
	\end{align}
	for $t\geq0$ with
	\begin{align*}
		\psi(t) =
		\begin{cases}
				{(1-e^{-b t})}/{b}, &\text{if }\ b\neq 0,\\
				t, &\text{if }\ b=0,
			\end{cases}
	\end{align*}
	due to \cite[Eq.~(4)]{yor}.
\end{rem}

\begin{rem}\label{rem:linear-growth}
	For every $1\leq p<\infty$ there exists a constant $C>0$ such that
	\begin{align}\label{eq:linear-growth}
		\left( \E{ \sup_{0\leq t\leq 1} \abs{X^x_t}^p} \right)^{1/p}
			\leq C\cdot (1+x)
	\end{align}
	for $x\geq0$, since the coefficients of SDE~\eqref{eq:cir} satisfy
	a linear growth condition, see, e.g., \cite[Thm.~2.4.4]{mao}.
\end{rem}

\begin{rem}[Monotonicity in the initial value]\label{rem:comparison-principle}
	The comparison principle for one-dimensional SDEs yields
	\begin{align*}
		X^x_t\leq X^y_t
	\end{align*}
	for all $t\geq 0$ and $0\leq x\leq y$, see, e.g., \cite[Thm.~IX.3.7]{revuz-yor}.
\end{rem}

% ------------------------------------------------------------------------
% ------------------------------------------------------------------------
% Section 4
% ------------------------------------------------------------------------
% ------------------------------------------------------------------------
\section{Strong Approximation}\label{sec:strong}
In this section we prove strong convergence rates for suitable
one-step approximation schemes for the SDE~\eqref{eq:cir}.
At first, we introduce the notion of a one-step scheme and identify
sufficient conditions for such a scheme to be strongly convergent.

% ------------------------------------------------------------------------
% ------------------------------------------------------------------------
A one-step scheme for the SDE~\eqref{eq:cir} with initial value $x\geq 0$
is a sequence of approximating processes $\YN=(\YN_t)_{t\geq 0}$
for $N\in\N$ that is defined by a Borel-measurable one-step function
\begin{align*}
	\oss\colon\R^+_0\times\left]0,1\right]\times\R\to\R^+_0
\end{align*}
in the following way. The process $\YN$ is defined at the grid of mesh size
$1/N$ by
\begin{align}\label{eq:oss}
	\YN_0=x \qquad\text{and}\qquad
		\YN_{(n+1)/N} = \oss\left( \YN_{n/N}, 1/N, \dWN_{n} \right)
\end{align}
for $n\in\N_0$, where
\begin{align*}
	\dWN_n = W_{(n+1)/N} - W_{n/N}.
\end{align*}
Moreover, between two grid points we use constant interpolation, i.e.,
\begin{align}\label{eq:oss-int}
	\YN_t = \YN_{n/N}, \qquad t\in{[n/N,(n+1)/N[},
\end{align}
for $n\in\N_0$. The value of $\YN_{n/N}$ for $n\in\N$ is thus a function
of the previous value $\YN_{(n-1)/N}$, the time increment and the Brownian increment.
Clearly, Euler and Milstein-type schemes are such one-step schemes.
We refer to \cite[Sec.~2.1.4]{hj2015} for one-step schemes
in a more general setting.

In case of SDE~\eqref{eq:cir} the classical Euler-Maruyama scheme
would be given by
\begin{align*}
	\oss(x,t,w) = x + (\delta-bx)\cdot t + 2\sqrt{x}\cdot w.
\end{align*}
However, this is not well-defined since $\oss(x,t,w)$ might be negative.
Hence classical schemes need to be appropriately modified to ensure
positivity if a square root term is used.

% ------------------------------------------------------------------------
% ------------------------------------------------------------------------
We consider the following conditions on one-step functions for $p\in{[1,\infty[}$.

\begin{enumerate}[({A}1)]
\item\label{ass:lipschitz}
	There exists a constant $K>0$ such that
	\begin{align*}
		\left( \E{ \abs{ \oss(x_1,t,W_t)-\oss(x_2,t,W_t) }^p } \right)^{1/p}
			\leq (1+Kt)\cdot \abs{x_1-x_2}
	\end{align*}
	for all $x_1,x_2\geq0$ and $t\in\left]0,1\right]$.
\end{enumerate}
We say that a one-step function $\oss$ satisfying (A\ref{ass:lipschitz}) is
$L_p$-Lipschitz continuous.

We define $\lerr\colon\R^+_0\times\left]0,1\right]\to\R^+_0$ by
\begin{align}\label{eq:loc-milstein}
	\lerr(x,t) =
		\begin{cases}
			t, &\text{if }\ x\leq t,\\
			t^{3/2}\cdot{{x^{-1/2}}}, &\text{if }\ t\leq x\leq1,\\
			t^{3/2}\cdot x, &\text{if }\ x\geq1.
		\end{cases}
\end{align}

\begin{enumerate}[({A}1)]\setcounter{enumi}{1}
\item\label{ass:local-error}
	There exists a constant $C>0$ such that
	\begin{align*}
		\left( \E{ \abs{ \oss(x,t,W_t)-X^x_t }^p } \right)^{1/p}
			\leq C\cdot \lerr(x,t)
	\end{align*}
	for all $x\geq0$ and $t\in\left]0,1\right]$.
\end{enumerate}
In Proposition~\ref{prop:milstein} we will show that the local discretization error
of a single Milstein-step is bounded by $\lerr$.
Hence we say that a one-step function $\oss$ satisfying (A\ref{ass:local-error})
has an $L_p$-Milstein-type local discretization error for the SDE~\eqref{eq:cir}.

\begin{rem}
	Under standard Lipschitz assumptions on the coefficients of an SDE, the Euler-Maruyama method
	and the Milstein method are $L_p$-Lipschitz continuous for all $p\in{[1,\infty[}$.
	Moreover, for all $p\in{[1,\infty[}$ the Euler-Maruyama method and the Milstein method
	have a local error of order $t$ and $t^{3/2}$, respectively, see, e.g.,
	\cite[Chap.~1.1]{milstein2004}.
\end{rem}

% ------------------------------------------------------------------------
% Section 4.1
% ------------------------------------------------------------------------
\subsection{$L_1$-convergence}\label{sec:L1-conv}
In this section we prove strong convergence rates with respect to the $L_1$-norm
for suitable one-step schemes. The reason why we restrict ourselves to $L_1$
in this section is indicated in Section~\ref{sec:regularity}.

% ------------------------------------------------------------------------
% ------------------------------------------------------------------------
\begin{prop}[Average local error]\label{prop:balance}
	There exists a constant $C>0$ such that
	\begin{align*}
		\E{ \lerr\left( X^x_s, t \right) }
			\leq C \cdot (1+x)\cdot t \cdot
			\begin{cases}
				1, &\text{if }\ s\leq t,\\
				(t/s)^{\min(1,\delta)/2}\cdot \left( 1+\ln(s/t)\cdot \ind{\set{1}}(\delta) \right),
						&\text{if }\ t\leq s,
			\end{cases}
	\end{align*}
	for all $t\in{]0,1]}$, $s\in{[0,1]}$, and $x\geq0$.
\end{prop}

\begin{proof}
	Consider the situation of Remark~\ref{rem:marginal} and let $c>0$.
	Observe that for $\varepsilon\in\left]0,c\right]$ we have
	\begin{align}\label{eq:chi1}
		\prob(Z\leq\eps) \preceq \eps^{\delta/2}
	\end{align}
	and
	\begin{align}\label{eq:chi2}
		\E{ Z^{-1/2}\cdot \ind{\set{Z\geq\eps}} } \preceq
			\begin{cases}
				1, &\text{if }\ \delta>1,\\
				1+\ln(c/\eps), &\text{if }\ \delta=1,\\
				\eps^{(\delta-1)/2}, &\text{if }\ \delta<1.
			\end{cases}
	\end{align}
	Furthermore, we clearly have
	\begin{align}\label{eq:scaling}
		\psi(s) \asymp s
	\end{align}
	for $s\in[0,1]$.
	
	For $s\leq t$ the claim follows from \eqref{eq:linear-growth} with $p=1$.
	In the following we consider the case $t\leq s$.
	Let $U_1,U_2\colon\R^+_0\times\left]0,1\right]\to\R^+_0$ be given by
	\begin{align*}
		U_1(x,t) = t^{3/2}\cdot x
	\end{align*}
	and
	\begin{align*}
		U_2(x,t) =
		\begin{cases}
			t, &\text{if }\ x\leq t,\\
			t^{3/2}\cdot{{x^{-1/2}}}, &\text{if }\ t\leq x,
		\end{cases}
	\end{align*}
	such that
	\begin{align}\label{eq:sum}
		\lerr(x,t) \leq U_1(x,t) + U_2(x,t).
	\end{align}
	Using \eqref{eq:linear-growth} with $p=1$ we obtain
	\begin{align}\label{eq:sum1}
		\E{ U_1( X^x_s, t) }
			\preceq t^{3/2}\cdot (1+x)
	\end{align}
	for $t\in{]0,1]}$, $s\in{[0,1]}$, and $x\geq0$.
	Combining \eqref{eq:marginal}, \eqref{eq:chi1}, \eqref{eq:chi2}, and \eqref{eq:scaling} yields
	\begin{align*}
		t\cdot \prob(X^0_s\leq t) = t\cdot \prob(Z\leq t/\psi(s))
			\preceq t\cdot (t/s)^{\delta/2}
	\end{align*}
	for $0<t\leq s\leq 1$ and
	\begin{align*}
		t^{3/2}\cdot \E{ (X^0_s )^{-1/2}  \cdot \ind{ \set{X^0_s\geq t} }}
			&= t^{3/2}\cdot (\psi(s))^{-1/2}\cdot \E{ Z^{-1/2}  \cdot \ind{ \set{Z\geq t/\psi(s)} }} \\
		&\preceq \frac{t^{3/2}}{\sqrt{s}} \cdot
			\begin{cases}
				1, &\text{if }\ \delta>1, \\
				1+\ln(s/t), &\text{if }\ \delta=1, \\
				\left(t/s\right)^{(\delta-1)/2}, &\text{if }\ \delta<1,
			\end{cases}
	\end{align*}
	for $0<t\leq s\leq 1$. Moreover, due to Remark~\ref{rem:comparison-principle} and
	monotonicity of $U_2(\cdot,t)$ we have
	\begin{align}\label{eq:sum2}
		\begin{aligned}[b]
			\E{ U_2(X^x_s,t) } &\leq \E{ U_2(X^0_s,t) } \\
			&\preceq t\cdot (t/s)^{\delta/2}+t\cdot
				\begin{cases}
					(t/s)^{1/2}, &\text{if }\ \delta>1, \\
					(t/s)^{1/2}\cdot (1+\ln(s/t)), &\text{if }\ \delta=1, \\
					(t/s)^{\delta/2}, &\text{if }\ \delta<1,
				\end{cases}
		\end{aligned}
	\end{align}
	for $0<t\leq s\leq 1$ and $x\geq0$.
	Combining \eqref{eq:sum}, \eqref{eq:sum1}, and \eqref{eq:sum2} completes the proof.
\end{proof}

% ------------------------------------------------------------------------
% ------------------------------------------------------------------------
\begin{thm}[$L_1$-convergence of one-step schemes]\label{thm:strong-L1}
	Let $\YN$ be a one-step scheme given by \eqref{eq:oss} and \eqref{eq:oss-int},
	and assume that (A\ref{ass:lipschitz}) and (A\ref{ass:local-error})
	are fulfilled for $p=1$. Then there exists a constant $C>0$ such that
	\begin{align*}
		\sup_{0\leq t\leq 1}\E{ \abs{X^x_t-\YN_t} } \leq C\cdot (1+x)\cdot
			\frac{ 1+\ind{\set{1}}(\delta)\cdot\ln N }{ N^{\min(1,\delta)/2} }
	\end{align*}
	for all $N\in\N$ and for all $x\geq 0$.
\end{thm}
\begin{proof}
	For notational convenience, we set
	\begin{align*}
		\XN_n = X^x_{n/N} \qquad\text{and}\qquad
		\YhatN_n = \YN_{n/N}
	\end{align*}
	for $n=0,\ldots,N$. Furthermore, we define
	\begin{align*}
		\eN_n = \E{ \abs{\XN_n-\YhatN_n} }
	\end{align*}
	for $n=0,\ldots,N$.
	Then we have $\eN_0=0$ and
	\begin{align*}
		\eN_{n+1} &\leq \E{ \abs{
				\XN_{n+1} - \oss\left( \XN_n, 1/N, \dWN_{n} \right)
			} } \\
		&\qquad+ \E{ \abs{
				\oss(\XN_n,1/N,\dWN_{n}) - \YhatN_{n+1} 
			} } \\
		&= \E { \E{ \abs{
				X_{1/N}^{\tilde{x}} - \oss\left( \tilde{x}, 1/N, W_{1/N} \right)
			} }_{\tilde{x}=\XN_n} } \\
		&\qquad+ \E{ \E{ \abs{
				\oss(\tilde{x},1/N,W_{1/N}) - \oss(\tilde{y},1/N,W_{1/N}) 
			} }_{(\tilde{x},\tilde{y})=(\XN_n,\YhatN_n)} }
	\end{align*}
	for $n=0,\ldots,N-1$. Using (A\ref{ass:local-error}) and (A\ref{ass:lipschitz})
	with $p=1$ we obtain
	\begin{align*}
		\eN_{n+1} \leq C_1\cdot \E{ \lerr(\XN_n,1/N) }
			+ \left( 1+K/N \right)\cdot \E{ \abs{\XN_n-\YhatN_n} }.
	\end{align*}
	Moreover, applying Proposition~\ref{prop:balance} yields
	\begin{align*}
		\eN_{n+1} &\leq \left( 1+K/N \right)\cdot \eN_n \\
		&\qquad+ C_1C_2\cdot (1+x) \cdot \frac{1}{N}\cdot
			\begin{cases}
				1, &\text{if }\ n=0,\\
				n^{-\min(1,\delta)/2}\cdot \left( 1+\ln(n)\cdot \ind{\set{1}}(\delta) \right),
						&\text{if }\ n\geq1,
			\end{cases}
	\end{align*}
	and hence
	\begin{align*}
		\eN_{n+1} \leq \left( 1+K/N \right)\cdot \eN_n
			+ 2\,C_1C_2\cdot (1+x) \cdot \frac{1}{N}\cdot
			\frac{ 1+\ln(N)\cdot\ind{\set{1}}(\delta) }{ (n+1)^{\min(1,\delta)/2} }
	\end{align*}
	for $n=0,\ldots,N-1$. Recursively, we get
	\begin{align*}
		\eN_{n} &\leq 2\,C_1C_2\cdot (1+x) \cdot \frac{1}{N}
			\cdot\left( 1+\ln(N)\cdot \ind{\set{1}}(\delta) \right)
			\cdot \sum_{k=1}^n \frac{ \left( 1+K/N \right)^{n-k} }{ k^{\min(1,\delta)/2} }
	\end{align*}
	and hence
	\begin{align*}
		\eN_{n} &\leq 2\,C_1C_2\,e^{K}\cdot (1+x) 
			\cdot\left( 1+\ln(N)\cdot \ind{\set{1}}(\delta) \right)
			\cdot \frac{1}{N}\sum_{k=1}^N k^{-\min(1,\delta)/2} \\
		&\leq 4\,C_1C_2\,e^{K}\cdot (1+x)
			\cdot\frac{ 1+\ln(N)\cdot \ind{\set{1}}(\delta) }{ N^{\min(1,\delta)/2} }
	\end{align*}
	for $n=0,\ldots,N$. This yields
	\begin{align}\label{eq:rate-grid}
		\max_{n=0,\ldots,N}\E{ \abs{\XN_n-\YhatN_n} } \leq 4\,C_1C_2\,e^{K}\cdot (1+x)\cdot
			\frac{ 1+\ind{\set{1}}(\delta)\cdot\ln N }{ N^{\min(1,\delta)/2} }
	\end{align}
	for all $N\in\N$ and for all $x\geq0$.
	Since the coefficients of SDE~\eqref{eq:cir} satisfy a linear
	growth condition we have
	\begin{align}\label{eq:sol-time-regularity}
		\E{ \abs{ X^x_s-X^x_t } } \preceq (1+x)\cdot\sqrt{\abs{s-t}}
	\end{align}
	for all $x\geq0$ and $s,t\in[0,1]$, see, e.g., \cite[Thm.~2.4.3]{mao}.
	Combining \eqref{eq:rate-grid} and \eqref{eq:sol-time-regularity}
	completes the proof.
\end{proof}

\begin{rem}
	Consider the proof of Theorem~\ref{thm:strong-L1}.
	An analysis of the global error by adding local errors is a
	classical technique for ordinary differential equations.
	For such a technique in the context of SDEs under standard
	Lipschitz assumptions we refer to \cite[Chap.~1.1]{milstein2004}.
	In case of SDE~\eqref{eq:cir}, it is crucial to control the average local
	error, see Proposition~\ref{prop:balance}.
\end{rem}

% ------------------------------------------------------------------------
% Section 4.2
% ------------------------------------------------------------------------
\subsection{$L_p$-convergence}\label{sec:Lp-conv}
In this section we extend the result from Theorem~\ref{thm:strong-L1}
to arbitrary $p>1$ by using interpolation of $L_p$-spaces. For this
we need the following additional assumption on a one-step scheme.
\begin{enumerate}[({A}1)]\setcounter{enumi}{2}
\item\label{ass:bounded}
	For every $1\leq q<\infty$ there exists a constant $C>0$ such that
	\begin{align*}
		\sup_{0\leq t\leq 1} \left( \E{ \abs{ \YN_t }^q } \right)^{1/q}
			\leq C\cdot (1+x)
	\end{align*}
	for all $x\geq0$ and $N\in\N$.
\end{enumerate}
We say that a one-step scheme $\YN$ satisfying (A\ref{ass:bounded})
is uniformly bounded.

\begin{rem}
	Under standard linear growth conditions on the coefficients of the SDE,
	the Euler-Maruyama method and the Milstein method are uniformly bounded,
	see, e.g., \cite[Lem.~2.7.1]{mao} for the Euler-Maruyama method and $q=2$.
\end{rem}

\begin{rem}[Interpolation of $L_p$-spaces]\label{rem:interpolation}
	Let $1\leq p<\infty$ and $0<\eps<1/p$. Set
	\begin{align}\label{eq:def-q}
		q = 1+\frac{1-1/p}{\eps}.
	\end{align}
	Note that $p\leq q<\infty$. An application of H\"older's inequality yields
	\begin{align*}
		\left( \E{ \abs{Z}^p } \right)^{1/p}
			= \left( \E{ \abs{Z}^{q\cdot(\eps p)}\cdot\abs{Z}^{(1/p-\eps)\cdot p} } \right)^{1/p}
			\leq \left( \E{ \abs{Z}^q } \right)^{\eps}
				\cdot \left( \E{ \abs{Z} } \right)^{1/p-\eps}
	\end{align*}
	for all random variables $Z$.
\end{rem}

\begin{cor}[$L_p$-convergence of one-step schemes]\label{cor:p}
	Consider the situation of Theorem~\ref{thm:strong-L1} and assume in addition that
	(A\ref{ass:bounded}) is fulfilled. Furthermore, let $1\leq p<\infty$ and $\eps>0$.
	Then there exists a constant $C>0$ such that
	\begin{align*}
		\sup_{0\leq t\leq 1}\left( \E{ \abs{X^x_t-\YN_t}^p } \right)^{1/p}
			\leq C\cdot(1+x)\cdot \frac{ 1}{ N^{\min(1,\delta)/(2p)-\eps} }
	\end{align*}
	for all $N\in\N$ and for all $x\geq0$.
\end{cor}
\begin{proof}
We may assume $\eps<1/p$. According to Remark~\ref{rem:interpolation}
	we have
	\begin{align*}
		\sup_{0\leq t\leq 1} &\left( \E{ \abs{X^x_t-\YN_t}^p } \right)^{1/p} \\
		&\leq \sup_{0\leq t\leq 1}\left( \E{ \abs{X^x_t-\YN_t} } \right)^{1/p-\eps} \\
		&\qquad\qquad\cdot \left(
				\sup_{0\leq t\leq 1}\left( \E{ \abs{X^x_t}^q } \right)^{1/q}
				+ \sup_{0\leq t\leq 1}\left( \E{ \abs{\YN_t}^q } \right)^{1/q}
			\right)^{q\cdot\eps}
	\end{align*}
	with $q$ given by \eqref{eq:def-q}. Using \eqref{eq:linear-growth} and (A\ref{ass:bounded})
	we get
	\begin{align*}
		\sup_{0\leq t\leq 1} &\left( \E{ \abs{X^x_t-\YN_t}^p } \right)^{1/p} 
			\preceq \sup_{0\leq t\leq 1}\left( \E{ \abs{X^x_t-\YN_t} } \right)^{1/p-\eps} 
			\cdot \left(1+x\right)^{q\cdot \eps}.
	\end{align*}
	It remains to apply Theorem~\ref{thm:strong-L1} and to observe that $1/p-\eps+q\eps=1$.
\end{proof}

% ------------------------------------------------------------------------
% Section 4.3
% ------------------------------------------------------------------------
\subsection{Regularity in the initial value}\label{sec:regularity}
In this section we study continuity properties of the solution of
SDE~\eqref{eq:cir} in the initial value. It turns out that the solution
is in general not locally Lipschitz continuous in the initial value
with respect to the $L_p$-norm if $p>1$.
This is the reason why we have restricted ourselves
in Section~\ref{sec:L1-conv} to the case $p=1$.
For results on local Lipschitz continuity in the initial value for more
general SDEs we refer to \cite{cox-hj}. However, in case of SDE~\eqref{eq:cir}
these results are restricted to $\delta\geq2$.

The following lemma implies that the solution of SDE~\eqref{eq:cir}
is Lipschitz continuous in the initial value with respect to the $L_1$-norm
on any compact time interval.
\begin{lem}\label{lem:initial-value}
	We have
	\begin{align*}
		\E{ \abs{X^x_t-X^y_t} } = e^{-bt}\cdot\abs{x-y}
	\end{align*}
	for all $x,y\geq0$ and for all $t\geq0$.
\end{lem}
\begin{proof}
	Let $x\geq y\geq0$. According to Remark~\ref{rem:comparison-principle} we have
	\begin{align*}
		\E{ \abs{X^x_t-X^y_t} } = \E{ X^x_t-X^y_t } = \E{ X^x_t } - \E{X^y_t }
	\end{align*}
	for all $t\geq0$. Using \eqref{eq:mean-sol} completes the proof.
\end{proof}

\begin{rem}
	The proof of Lemma~\ref{lem:initial-value}
	is a general technique to obtain $L_1$-Lipschitz continuity
	in the initial value for a large class of one-dimensional SDEs.
	The comparison principle reduces this problem to the computation
	of expected values.
\end{rem}

\begin{ex}[One-dimensional squared Bessel process]\label{ex:bessel}
	In \cite{bessel1}, it was shown that for $\delta=1$ and $b=0$ we have
	\begin{align*}
		X^x_t = \left( W_t+\sqrt{x} - \min\left(
				0, \inf_{0\leq s\leq t} W_s+\sqrt{x}
			\right) \right)^2
	\end{align*}
	for $t\geq0$ and $x\geq 0$. This yields
	\begin{align*}
		\abs{ X_t^x-X_t^0 } =
			\begin{cases}
				0, &\text{if }\inf_{0\leq s\leq t} W_s+\sqrt{x}\leq 0,\\
				(W_t+\sqrt{x})^2-(W_t-\inf_{0\leq s\leq t} W_s)^2,
					&\text{if }\inf_{0\leq s\leq t} W_s+\sqrt{x}\geq 0.
			\end{cases}
	\end{align*}
	Using this one can show
	\begin{align*}
		\left(\E{ \abs{X_t^x-X_t^0}^p }\right)^{1/p} \asymp x^{(1+1/p)/2}
	\end{align*}
	for $x\in[0,1]$, where $1\leq p<\infty$ and $t>0$.
\end{ex}

In the rest of this section we assume
\begin{align*}
	0<\delta<2 \qquad\text{and}\qquad b=0,
\end{align*}
i.e., the solution of SDE~\eqref{eq:cir} is a squared Bessel process of dimension $\delta$.
For $1\leq p<\infty$ define the maximal H\"older exponent by
\begin{align*}
	\exin(\delta,p) = \sup\left\{\alpha\geq 0:\ \exists C>0\ \ \forall x\in \left]0,1\right]:
			\ \left(\E{ \abs{X^x_1-X^0_1}^p }\right)^{1/p} \leq C\cdot x^\alpha
		\right\}.
\end{align*}
Note that replacing the time point $t=1$ in the definition of $\exin$
to an arbitrary time point $t>0$ does not affect the value of $\exin$,
see Remark~\ref{rem:scaling-space} and Remark~\ref{rem:scaling-time}.
From Lemma~\ref{lem:initial-value} and Example~\ref{ex:bessel} we already have
\begin{align}\label{eq:exin-1}
	\exin(\delta,1) = 1
\end{align}
and
\begin{align}\label{eq:exin-2}
	\exin(1,p) = \left(1+1/p\right)/2.
\end{align}

\begin{prop}\label{prop:initial-value}
	We have
	\begin{align}\label{eq:exin}
		1/p \leq \exin(\delta,p) \leq 1/p+\delta/2-\delta/(2p).
	\end{align}
	In particular, we have $\exin(\delta,p)<1$ if and only if $p>1$.
\end{prop}
\begin{proof}
	If $p=1$, then \eqref{eq:exin} follows from \eqref{eq:exin-1}.
	Note that for $x\in\left]0,1\right]$ we have
	\begin{align}\label{eq:prob-positive}
		\prob\left( \forall t\in[0,1]:\ X_{t}^{x}>0 \right)
			\asymp x^{1-\delta/2}.
	\end{align}
	This follows from \cite[p.~75]{handbook},
	where the density of the first hitting time of zero is given for Bessel processes.
	Let $1<p<\infty$ and let $1<q<\infty$ be the dual of $p$, i.e.,
	$1/p+1/q=1$. Using H\"older's inequality, Remark~\ref{rem:comparison-principle},
	Lemma~\ref{lem:initial-value}, and \eqref{eq:prob-positive} we get
	\begin{align*}
		\left(\E{ \abs{X^x_1-X^0_1}^p }\right)^{1/p}
			&\geq \E{ \abs{X^x_1-X^0_1}\cdot\ind{\set{\forall t\in[0,1]:X^x_t>0}} }
			\cdot \big(\, \prob\left(\forall t\in[0,1]:X^x_t>0\right) \big)^{-1/q} \\
		&= \E{ \abs{X^x_1-X^0_1} }
			\cdot \big(\, \prob\left(\forall t\in[0,1]:X^x_t>0\right) \big)^{-1/q} \\
		&\asymp x\cdot x^{-1/q\cdot (1-\delta/2)}
	\end{align*}
	for $x\in\left]0,1\right]$,
	which shows the upper bound in \eqref{eq:exin}. The lower bound follows by
	interpolation using Remark~\ref{rem:interpolation}, Lemma~\ref{lem:initial-value}
	and \eqref{eq:linear-growth}, cf.~the proof of Corollary~\ref{cor:p}.
\end{proof}

\begin{rem}
	Let us mention that the upper bound in \eqref{eq:exin} is sharp
	for $p=1$ and $\delta=1$, see \eqref{eq:exin-1} and \eqref{eq:exin-2}.
\end{rem}

% ------------------------------------------------------------------------
% ------------------------------------------------------------------------
% Section 5
% ------------------------------------------------------------------------
% ------------------------------------------------------------------------
\section{Tamed Milstein Scheme}\label{sec:scheme}
In this section we introduce a truncated Milstein scheme and prove
that this scheme satisfies the assumptions of
Theorem~\ref{thm:strong-L1} and Corollary~\ref{cor:p}.

% ------------------------------------------------------------------------
% ------------------------------------------------------------------------
Consider $\mil\colon\R^+_0\times\left]0,1\right]\times\R\to\R$ given by
\begin{align*}
	\mil(x,t,w) &= x + (\delta-bx)\cdot t + 2\sqrt{x}\cdot w
		+ \left(w^2-t\right) \\
	&= \left( \sqrt{x}+w \right)^2
		+ \left( \delta-1-bx \right)\cdot t.
\end{align*}
This function models a single Milstein-step. Moreover, it only
preserves positivity if $b\leq0$ and $\delta\geq1$.
Hence it is not a valid one-step scheme in general.
However, the analysis of $\mil$ is an important step since we will
construct valid one-step schemes that are close to $\mil$.

\begin{prop}[Error of a Milstein-step]\label{prop:milstein}
	For every $1\leq p<\infty$ there exists a constant $C>0$ such that
	\begin{align*}
		\left( \E{ \abs{ \mil(x,t,W_t)-X^x_t }^p } \right)^{1/p}
			\leq C\cdot \lerr(x,t)
	\end{align*}
	for all $x\geq0$ and $t\in\left]0,1\right]$, where $\lerr$ is
	given by \eqref{eq:loc-milstein}.
\end{prop}

The proof of Proposition~\ref{prop:milstein} exploits the following
simple lemma, which is a refinement of \eqref{eq:linear-growth}.
\begin{lem}
	For every $1\leq p<\infty$ there exists a constant $C>0$ such that
	\begin{align}\label{eq:sol-sup}
		\left( \E{ \sup_{0\leq s\leq t} \abs{X^x_s}^p}\right)^{1/p}
			\leq C\cdot (x+t)
	\end{align}
	for all $x\geq0$ and $t\in[0,1]$.
\end{lem}
\begin{proof}
	According to \eqref{eq:mean-sol} we have
	\begin{align*}
		\E{X^x_t} \preceq x+t
	\end{align*}
	for $x\geq0$ and $t\in[0,1]$. Combining this with \eqref{eq:linear-growth}
	and a Burkholder-Davis-Gundy-type inequality~\cite[Thm.~1.7.2]{mao}
	we obtain
	\begin{align*}
		\left( \E{ \sup_{0\leq s\leq t} \abs{X^x_s}^2}\right)^{1/2}
			\preceq x + (1+x)\,t + \sqrt{t}\cdot\sqrt{x+t}
	\end{align*}
	for $x\geq0$ and $t\in[0,1]$, i.e., \eqref{eq:sol-sup} holds for $p=2$.
	In the same way we get \eqref{eq:sol-sup} for $p=4,8,16,\ldots$,
	which suffices.
\end{proof}

\begin{proof}[Proof of Proposition~\ref{prop:milstein}]
	We may assume $2\leq p<\infty$. From \eqref{eq:linear-growth},
	a Burkholder-Davis-Gundy-type inequality~\cite[Thm.~1.7.2]{mao},
	and \eqref{eq:sol-sup} we get
	\begin{align}\label{eq:2}
		\left( \E{ \sup_{0\leq s\leq t} \abs{X^x_s-x}^p } \right)^{1/p}
			\preceq (1+x)\,t + \sqrt{t}\cdot\sqrt{x+t}
			\preceq (1+x)\,t + \sqrt{xt}
	\end{align}
	for $x\geq0$ and $t\in[0,1]$. Moreover, we obtain
	\begin{align}\label{eq:3}
		\left( \E{ \abs{\mil(x,t,W_t)-x}^p }\right)^{1/p}
			\preceq (1+x)\,t + \sqrt{xt}+t
			\preceq (1+x)\,t + \sqrt{xt}
	\end{align}
	for $x\geq0$ and $t\in\left]0,1\right]$. Combining \eqref{eq:2} and \eqref{eq:3}
	yields the claim for $x\leq t$.
	Furthermore, according to \eqref{eq:2} the error of the drift term satisfies
	\begin{align*}
		\left( \E{ \abs{
				\int_0^t (\delta-bX^x_s)\,\dd s - \int_0^t (\delta-bx)\,\dd s
			}^p } \right)^{1/p}
			&\preceq (1+x)\,t^2 + \sqrt{x}\cdot t^{3/2} \\
		&\preceq \lerr(x,t)
	\end{align*}
	for $x\geq0$ and $t\in\left]0,1\right]$.
	
	Define the stopping time
	\begin{align*}
		\tau^x = \inf \left\{s\geq 0:\,\abs{X^x_s-x}=x/2 \right\}
	\end{align*}
	for $x\geq 0$. Using Markov's inequality we get
	\begin{align*}
		\prob\left(\tau^x\leq t\right)
			= \prob\left( \sup_{0\leq s\leq t} \abs{X^x_s-x}\geq x/2 \right) 
			\leq \frac{ \E{ \sup_{0\leq s\leq t}\abs{X^x_s-x}^p } }{(x/2)^p} 
	\end{align*}
	for $t\geq0$ and $x>0$, and hence \eqref{eq:2} implies
	\begin{align*}
		\prob\left(\tau^x\leq t\right)
			\preceq \frac{(xt)^p+(xt)^{p/2}}{x^p}
			=t^p+(t/x)^{p/2}
			\leq t^{p/2}+(t/x)^{p/2}
			\leq 2\cdot\frac{ t^{p/2} }{ \min(1,x^{p/2}) }
	\end{align*}
	for $x\geq t$ and $t\in\left]0,1\right]$.
	By quadrupling $p$ we obtain
	\begin{align}\label{eq:hitting}
		\prob\left(\tau^x\leq t\right)
			\preceq \frac{ t^{2p} }{ \min(1,x^{2p}) }
	\end{align}
	for $x\geq t$ and $t\in\left]0,1\right]$.
	Define the stopped process
	$\tilde X^x=(\tilde X^x_t)_{t\geq 0}$ by
	\begin{align*}
		\tilde X^x_t = X^x_{t\wedge\tau^x}
	\end{align*}
	for $x\geq 0$.
	Clearly, $\tilde X^x_t\in[x/2,3x/2]$ for $t\geq 0$.
	It\^{o}'s lemma shows
	\begin{align*}
		\sqrt{\tilde X^x_t} = \sqrt{x} + \int_0^{t\wedge\tau^x}
			\frac{ (\delta-1)-b\tilde X^x_s }{ 2\sqrt{\tilde X^x_s} }\,\dd s
			+ W_{t\wedge\tau^x}
	\end{align*}
	and hence
	\begin{align}\label{eq:4}
		\left( \E{ \abs{ \sqrt{\tilde X^x_t}-\left(\sqrt{x}+W_{t\wedge\tau^x}\right) }^p
			}\right)^{1/p}
			\preceq (1+x)\,t/\sqrt{x}
	\end{align}
	for $t\geq0$ and $x>0$.
	Combining the Cauchy-Schwarz inequality with \eqref{eq:sol-sup}
	and \eqref{eq:hitting} yields
	\begin{align*}
		&\left( \E{ \abs{ \sqrt{X^x_t}-\left(\sqrt{x}+W_t\right) }^p
				\cdot \ind{ \set{\tau^x\leq t} }
			} \right)^{1/p} \\
		&\qquad\leq \left( \E{ \abs{ \sqrt{X^x_t}-\left(\sqrt{x}+W_t\right) }^{2p} }
			\right)^{1/(2p)}
			\cdot \Big( \prob(\tau^x\leq t) \Big)^{1/(2p)} \\
		&\qquad\preceq \left(\sqrt{x+t}+\sqrt{x}+\sqrt{t}\right)\cdot \frac{t}{\min(1,x)}
	\end{align*}
	and hence
	\begin{align}\label{eq:5}
		\left( \E{ \abs{\sqrt{X^x_t}-\left(\sqrt{x}+W_t\right)}^p
				\cdot \ind{ \set{\tau^x\leq t} }
			} \right)^{1/p}
			\preceq t\cdot (\sqrt{x}+1/\sqrt{x})
	\end{align}
	for $x\geq t$ and $t\in\left]0,1\right]$. Moreover, we have
	\begin{align*}
		\left( \E{ \abs{\sqrt{X^x_t}-\left(\sqrt{x}+W_t\right)}^p } \right)^{1/p}
			&\leq \left( \E{ \abs{\sqrt{X^x_t}-\left(\sqrt{x}+W_t\right)}^p
				\cdot \ind{ \set{\tau^x>t} }
			}\right)^{1/p} \\
		&\qquad+ \left( \E{ \abs{\sqrt{X^x_t}-\left(\sqrt{x}+W_t\right)}^p
				\cdot \ind{ \set{\tau^x\leq t} }
			}\right)^{1/p}
	\end{align*}
	such that \eqref{eq:4} and \eqref{eq:5} yield
	\begin{align}\label{eq:6}
		\left( \E{ \abs{\sqrt{X^x_t}-\left(\sqrt{x}+W_t\right)}^p } \right)^{1/p}
			\preceq t\cdot (\sqrt{x}+1/\sqrt{x})
	\end{align}
	for $x\geq t$ and $t\in\left]0,1\right]$.
	Using a Burkholder-Davis-Gundy-type inequality~\cite[Thm.~1.7.2]{mao}
	and \eqref{eq:6} we obtain
	\begin{align*}
		&\left( \E{ \abs{ \int_0^t \sqrt{X^x_s}\,\dd W_s
				-\int_0^t \left(\sqrt{x}+W_s\right)\dd W_s
			}^p } \right)^{1/p} \\
		&\qquad\preceq \sqrt{t} \cdot \sup_{0\leq s\leq t}
			\left( \E{ \abs{\sqrt{X^x_s}-\left(\sqrt{x}+W_s\right)}^p } \right)^{1/p} \\
		&\qquad\preceq \lerr(x,t) 
	\end{align*}
	for $x\geq t$ and $t\in\left]0,1\right]$.
\end{proof}

% ------------------------------------------------------------------------
% Section 5.1
% ------------------------------------------------------------------------
\subsection{$L_1$-convergence}\label{sec:L1-mil}
Recall that $\mil$ is given by
\begin{align*}
	\mil(x,t,w) = h(x,t,w) + \left( \delta-1-bx \right)\cdot t
\end{align*}
with
\begin{align*}
	h(x,t,w)=\left( \sqrt{x}+w \right)^2.
\end{align*}
Consider the one-step function
$\tmil\colon\R^+_0\times\left]0,1\right]\times\R\to\R^+_0$ given by
\begin{align}\label{eq:truncated-scheme}
	\tmil(x,t,w) = \left(
			\htil(x,t,w) + \left( \delta-1-bx \right)\cdot t
		\right)^+
\end{align}
with
\begin{align*}
	\htil(x,t,w)=\left( \max\left(
			\sqrt{t},\sqrt{\max(t,x)}+w
		\right) \right)^2.
\end{align*}
The corresponding one-step scheme is denoted by $\Ytmil$.
We refer to this scheme as truncated Milstein scheme.
Let us mention that we have separated the nonlinear parts of $\mil$
and $\tmil$, namely $h$ and $\htil$, from
the linear drift term.
The following lemma shows that $\htil$ is $L_1$-Lipschitz continuous
with constant $1$ and that $\htil$ is close to $h$ in a suitable way,
cf.~(A\ref{ass:lipschitz}) and (A\ref{ass:local-error}).
The proof of Lemma~\ref{lem:lipschitz-part} is postponed to the appendix.

\begin{lem}\label{lem:lipschitz-part}
	We have
	\begin{align*}
		\E{ \abs{ \htil(x_1,t,W_t)-\htil(x_2,t,W_t) } } \leq \abs{x_1-x_2}
	\end{align*}
	for all $x_1,x_2\geq0$ and $t\in\left]0,1\right]$.
	Furthermore, for every $1\leq p<\infty$ there exists a constant $C>0$ such that
	\begin{align*}
		\left(\E{ \abs{ h(x,t,W_t)-\htil(x,t,W_t) }^p }\right)^{1/p} \leq C\cdot \lerr(x,t)
	\end{align*}
	for all $x\geq0$ and $t\in\left]0,1\right]$.
\end{lem}

The following theorem shows that the truncated Milstein scheme satisfies
the assumptions of Theorem~\ref{thm:strong-L1}.
\begin{thm}[$L_1$-convergence of truncated Milstein scheme]\label{thm:truncated-milstein}
	There exists a constant $K>0$ such that
	\begin{align*}
		\E{ \abs{ \tmil(x_1,t,W_t)-\tmil(x_2,t,W_t) } }
			\leq (1+Kt)\cdot \abs{x_1-x_2}
	\end{align*}
	for all $x_1,x_2\geq0$ and $t\in\left]0,1\right]$, i.e., (A\ref{ass:lipschitz})
	is fulfilled for $p=1$.
	Furthermore, for every $1\leq p<\infty$ there exists a constant $C>0$ such that
	\begin{align*}
		\left(\E{ \abs{ \tmil(x,t,W_t)-X^x_t }^p }\right)^{1/p}
			\leq C\cdot \lerr(x,t)
	\end{align*}
	for all $x\geq0$ and $t\in\left]0,1\right]$, i.e., (A\ref{ass:local-error})
	is fulfilled for every $1\leq p<\infty$.
\end{thm}
\begin{proof}
	Using $\abs{y^+-z^+}\leq\abs{y-z}$ for $y,z\in\R$ we obtain
	\begin{align*}
		&\E{ \abs{ \tmil(x_1,t,W_t)-\tmil(x_2,t,W_t) } } \\
		&\qquad\leq\E{ \abs{
				\htil(x_1,t,W_t)-\htil(x_2,t,W_t) + (x_2-x_1)\cdot bt
			} } \\
		&\qquad\leq (1+bt)\cdot \abs{x_1-x_2}
	\end{align*}
	due to the first part of Lemma~\ref{lem:lipschitz-part}.
	Moreover, using $\abs{z^+-y}\leq\abs{z-y}$ for $y\geq0$ and $z\in\R$ we have
	\begin{align*}
		&\left(\E{ \abs{ \tmil(x,t,W_t)-X^x_t }^p }\right)^{1/p} \\
		&\qquad\leq \left(\E{ \abs{
				\htil(x,t,W_t) + (\delta-1-bx)\cdot t-X^x_t
			}^p }\right)^{1/p} \\
		&\qquad\leq \left(\E{ \abs{ \mil(x,t,W_t)-X^x_t }^p }\right)^{1/p}
			+ \left(\E{ \abs{ h(x,t,W_t)-\htil(x,t,W_t) }^p }\right)^{1/p}.
	\end{align*}
	Applying Proposition~\ref{prop:milstein} and the second part of Lemma~\ref{lem:lipschitz-part}
	yields the second statement.
\end{proof}

\begin{rem}
	Let us stress that there is some freedom regarding the particular
	form of the truncated Milstein scheme. For instance, the proof
	of Theorem~\ref{thm:truncated-milstein} shows that the positive
	part in \eqref{eq:truncated-scheme} may be replaced by the absolute
	value.
\end{rem}

% ------------------------------------------------------------------------
% Section 5.2
% ------------------------------------------------------------------------
\subsection{$L_p$-convergence}\label{sec:Lp-mil}
In this section we show that the truncated Milstein scheme $\Ytmil$ defined by
the one-step function $\tmil$ given in \eqref{eq:truncated-scheme} is uniformly
bounded and hence satisfies the assumptions of Corollary~\ref{cor:p}.

\begin{prop}\label{lem:bounded}
	For every $1\leq q<\infty$ there exists a constant $C>0$ such that
	\begin{align*}
		\sup_{0\leq t\leq 1} \left( \E{ \abs{ \Ytmil_t }^q } \right)^{1/q}
			\leq C\cdot (1+x)
	\end{align*}
	for all $x\geq0$ and $N\in\N$, i.e., (A\ref{ass:bounded})
	is fulfilled.
\end{prop}
\begin{proof}
	Let $x\geq0$, $t\in\left]0,1\right]$, and $w\in\R$. At first, note that
	\begin{align*}
		\htil(x,t,w) &\leq
			\left( \max\left(\sqrt{t},\sqrt{t}+w\right) \right)^2
			+ \left( \max\left(\sqrt{t},\sqrt{x}+w\right) \right)^2 \\
		&\leq 2t + 2 w^2 + t + \left(\sqrt{x}+w\right)^2
	\end{align*}
	and
	\begin{align*}
		\left( \delta-1-bx \right)\cdot t \leq ( \delta+\abs{b}x ) \cdot t.
	\end{align*}
	Moreover, note that $\max\left(\sqrt{t},\sqrt{x}+w\right)$
	is monotonically increasing in $x$.
	Hence the auxiliary one-step function
	$g\colon\R^+_0\times\left]0,1\right]\times\R\to\R^+_0$ defined by
	\begin{align*}
		g(x,t,w) = x + ( \delta+\abs{b}x+3 ) \cdot t
			+ 3w^2 + 2\sqrt{x}\cdot w
	\end{align*}
	satisfies
	\begin{align*}
		\tmil(x_1,t,w)\leq g(x_2,t,w).
	\end{align*}
	for $0\leq x_1\leq x_2$. This yields
	\begin{align}\label{eq:aux-scheme-bound}
		0\leq \Ytmil_t \leq \ZN_t
	\end{align}
	for all $t\geq0$, where $\ZN$ denotes the corresponding auxiliary scheme
	with piecewise constant interpolation. Moreover, the auxiliary scheme
	satisfies the integral equation
	\begin{align*}
		\ZN_t = x + \int_0^t \left( \delta+\abs{b}\ZN_s+6 \right) \dd s
			+ \int_0^t \left(
				2\sqrt{\ZN_s} + 6\,(W_s-\Wbar_s)
			\right) \dd W_s
	\end{align*}
	for $t=0,1/N,2/N,\ldots$, where $(\Wbar_t)_{t\geq0}$
	denotes the piecewise constant interpolation of $W$ at the grid of
	mesh size $1/N$, i.e., $\Wbar_t=W_{\lfloor tN\rfloor/N}$.
	Due to this integral equation we can now apply standard techniques
	exploiting the linear growth condition to obtain uniform boundedness,
	cf.~\cite[Lem.~2.6.1]{mao}. Let $2\leq q<\infty$. Straightforward
	calculations show
	\begin{align*}
		\abs{\ZN_s}^q &\preceq 1 + x^q + \int_0^s \abs{\ZN_u}^q \dd u \\
		&\qquad\qquad+
			\abs{ \int_0^s \left(2\sqrt{\ZN_u}+6\,(W_u-\Wbar_u)\right) \dd W_u }^q
	\end{align*}
	for $x\geq0$ and $s=0,1/N,\ldots,1$. This yields
	\begin{align*}
		\sup_{0\leq s\leq t} \abs{\ZN_s}^q &\preceq 1 + x^q + \int_0^t \abs{\ZN_u}^q \dd u \\
		&\qquad\qquad+ \sup_{0\leq s\leq t}
			\abs{ \int_0^s \left(2\sqrt{\ZN_u}+6\,(W_u-\Wbar_u)\right) \dd W_u }^q
	\end{align*}
	for $x\geq0$ and $t\in[0,1]$.
	Using a Burkholder-Davis-Gundy-type inequality \cite[Thm.~1.7.2]{mao} and
	the linear growth condition we obtain
	\begin{align*}
		\E{ \sup_{0\leq s\leq t} \abs{\ZN_s}^q } &\preceq 1 + x^q
			+ \int_0^t \E{ \sup_{0\leq s\leq u} \abs{\ZN_s}^q } \dd u \\
		&\qquad\qquad+ \int_0^t \E{ \abs{2\sqrt{\ZN_u}+6\,(W_u-\Wbar_u)}^q } \dd u \\
		&\preceq 1 + x^q
			+ \int_0^t \E{ \sup_{0\leq s\leq u} \abs{\ZN_s}^q } \dd u
	\end{align*}
	for $x\geq0$ and $t\in[0,1]$.
	Applying Gronwall's lemma yields the desired inequality for the
	auxiliary scheme. Due to \eqref{eq:aux-scheme-bound} this
	inequality also holds for the truncated Milstein scheme.
\end{proof}

% ------------------------------------------------------------------------
% ------------------------------------------------------------------------
% Appendix
% ------------------------------------------------------------------------
% ------------------------------------------------------------------------
\section*{Appendix}

\begin{lem}\label{lem:lipschitz-normal}
	Let $Z$ be standard normally distributed. We have
	\begin{align*}
		\E{ \abs{ \left(\max(1,\sqrt{x_1}+Z)\right)^2
				-\left(\max(1,\sqrt{x_2}+Z)\right)^2
			} } \leq \abs{x_1-x_2}
	\end{align*}
	for all $x_1,x_2\geq1$.
	Furthermore, for every $1\leq p<\infty$ there exists a constant $C>0$ such that
	\begin{align*}
		\left(\E{ \abs{ \left(\max(1,\sqrt{x}+Z)\right)^2
				-\left(\sqrt{x}+Z\right)^2
			}^p }\right)^{1/p} \leq C\cdot \frac{1}{\sqrt{x}}
	\end{align*}
	for all $x\geq 1$.
\end{lem}
\begin{proof}
	Let $\phi$ and $\Phi$ denote the Lebesgue density and the
	distribution function of the standard normal distribution,
	respectively. Moreover, let $f\colon\R\to\R$ be given by
	\begin{align*}
		f(x) &= \E{ \left(\max\left(1,x+Z\right)\right)^2 } \\
		&= \Phi(1-x) + x^2 \Phi(x-1)
			+ 2x\int_{1-x}^\infty z\phi(z)\,\dd z
			+ \int_{1-x}^\infty z^2\phi(z)\,\dd z \\
		&= \Phi(1-x) + x^2 \Phi(x-1)
			+ 2x\,\phi(1-x) \\
		&\qquad\qquad+ 1+(1-x)\,\phi(1-x)-\Phi(1-x) \\
		&= 1+ (1+x)\,\phi(1-x) + x^2\Phi(x-1).
	\end{align*}
	Hence the derivative of $f$ reads
	\begin{align*}
		f'(x) &= \phi(1-x) - (1+x)\,\phi'(1-x)
			+2x\,\Phi(x-1) + x^2\phi(x-1) \\
		&= 2\phi(1-x) + 2x\,\Phi(x-1).
	\end{align*}
	For $x>0$ we define $g(x)=f(\sqrt{x})$ such that
	\begin{align*}
		g'(x) = \frac{1}{\sqrt{x}}\cdot \phi(\sqrt{x}-1) + \Phi(\sqrt{x}-1).
	\end{align*}
	Using
	\begin{align*}
		\frac{1}{x+1} \leq \frac{ \sqrt{4+x^2}-x }{2}
	\end{align*}
	for $x\geq0$ and \cite{birnbaum}, we have
	\begin{align*}
		\frac{1}{x+1}\cdot\phi(x)+\Phi(x)
			\leq \frac{ \sqrt{4+x^2}-x }{2}\cdot\phi(x)+\Phi(x) \leq 1
	\end{align*}
	for $x\geq0$. This yields
	\begin{align*}
		g'(x)\leq 1
	\end{align*}
	for $x\geq1$. For $x_1\geq x_2\geq1$ we hence get
	\begin{align*}
		\E{ \abs{ \left(\max(1,\sqrt{x_1}+Z)\right)^2
				-\left(\max(1,\sqrt{x_2}+Z)\right)^2
			} } 
			= g(x_1)-g(x_2) \leq x_1-x_2,
	\end{align*}
	which shows the first claim.
	
	The Cauchy-Schwarz inequality implies
	\begin{align*}
		&\E{ \abs{ \left(\max(1,\sqrt{x}+Z)\right)^2
				-\left(\sqrt{x}+Z\right)^2
			}^p } \\
		&\qquad= \E{ \abs{ 1-\left(\sqrt{x}+Z\right)^2 }^p
			\cdot \ind{\set{\sqrt{x}+Z<1}} } \\
		&\qquad\leq \sqrt{ \E{ \abs{ 1-\left(\sqrt{x}+Z\right)^2 }^{2p} } }
			\cdot \sqrt{ \prob(\sqrt{x}+Z<1) }.
	\end{align*}
	Hence we get
	\begin{align*}
		&\left( \E{ \abs{ \left(\max(1,\sqrt{x}+Z)\right)^2
				-\left(\sqrt{x}+Z\right)^2
			}^p } \right)^{1/p} \\
		&\qquad\leq \left( 1+2x+2\cdot \left(\E{(Z^2)^{2p}}\right)^{1/(2p)}
			\right) \cdot \left(\prob(Z>\sqrt{x}-1)\right)^{1/(2p)} \\
		&\qquad\preceq x\cdot \left(\prob(Z>\sqrt{x}-1)\right)^{1/(2p)}
	\end{align*}
	for $x\geq 1$. By using a standard tail estimate for the standard
	normal distribution, see, e.g., \cite[p.~31]{revuz-yor}, we get the second claim.
\end{proof}

\begin{rem}
	The upper bound in the second statement in Lemma~\ref{lem:lipschitz-normal}
	can be considerably improved due to the exponential decay of the density of the
	standard normal distribution. However, $1/\sqrt{x}$ suffices for
	our purposes.
\end{rem}

\begin{proof}[Proof of Lemma~\ref{lem:lipschitz-part}]
	At first, note that
		\begin{align}\label{eq:normal-scaling1}
		h(x,t,w) = t\cdot \left(\sqrt{x/t}+{w}/{\sqrt{t}}
		\right)^2
	\end{align}
	and
	\begin{align}\label{eq:normal-scaling2}
		\htil(x,t,w) = t\cdot \left(\max\left(
			1,\sqrt{\max\left(1,x/t\right)}+{w}/{\sqrt{t}}
		\right)\right)^2.
	\end{align}
	Using \eqref{eq:normal-scaling2} and the first part of Lemma~\ref{lem:lipschitz-normal}
	we obtain 
	\begin{align*}
		\E{ \abs{ \htil(x_1,t,W_t)-\htil(x_2,t,W_t) } }
			\leq t\cdot \abs{\max\left(1,x_1/t\right)-\max\left(1,x_2/t\right)}
			\leq \abs{x_1-x_2}. 
	\end{align*}
	Moreover, using \eqref{eq:normal-scaling1}, \eqref{eq:normal-scaling2}, and
	the second part of Lemma~\ref{lem:lipschitz-normal} we obtain 
	\begin{align*}
		\left(\E{ \abs{ h(x,t,W_t)-\htil(x,t,W_t)}^p }\right)^{1/p}
			\leq t\cdot C\cdot\frac{1}{\sqrt{x/t}}
			\preceq \lerr(x,t)
	\end{align*}
	for $x\geq t$. Finally, we have
	\begin{align*}
		&\left( \E{ \abs{h(x,t,W_t)-\htil(x,t,W_t)}^p } \right)^{1/p} \\
		&\qquad\leq \left( \E{ \abs{ h(x,t,W_t)}^p } \right)^{1/p}
			+ \left( \E{ \abs{ \htil(x,t,W_t)}^p }\right)^{1/p}
			\preceq t = \lerr(x,t)
	\end{align*}
	for $x\leq t$.
\end{proof}

% ------------------------------------------------------------------------
% ------------------------------------------------------------------------
% Acknowledgement
% ------------------------------------------------------------------------
% ------------------------------------------------------------------------
\section*{Acknowledgement}
We thank Thomas M\"uller-Gronbach and Klaus Ritter for valuable discussions
and comments.

% ------------------------------------------------------------------------
% ------------------------------------------------------------------------
% Bibliography
% ------------------------------------------------------------------------
% ------------------------------------------------------------------------
\bibliographystyle{plain}
\renewcommand*{\bibname}{References}
\bibliography{references}

\end{document}